\documentclass[10pt,brazil,english]{amsart}

\usepackage[T1]{fontenc}
\usepackage[latin9]{inputenc}
\usepackage[letterpaper]{geometry}
\usepackage{amsthm}
\usepackage{amsmath}
\usepackage{setspace}
\usepackage{amssymb}
\usepackage{esint}
\usepackage{hyperref}

\makeatletter
\numberwithin{equation}{section}
\numberwithin{figure}{section}

\theoremstyle{plain}
\newtheorem{thm}{Theorem}
  \theoremstyle{definition}
  \newtheorem{defn}[thm]{Definition}
  \theoremstyle{plain}
  \newtheorem{prop}[thm]{Proposition}
  \theoremstyle{plain}
  \newtheorem{cor}[thm]{Corollary}
  \theoremstyle{remark}
  \newtheorem{rem}[thm]{Remark}
 \theoremstyle{definition}
  \newtheorem{example}[thm]{Example}
  \theoremstyle{plain}
  \newtheorem{lem}[thm]{Lemma}

\makeatother

\usepackage{babel}

\begin{document}

\title{The regularity of the $\eta$ function for the Shubin calculus}
\author{Pedro T. P. Lopes}
\email{pplopes@ime.usp.br; dritao@yahoo.com}
\address{Instituto de Matem\'atica e Estat\'istica, Universidade de S\~ao Paulo,  05508-090, S\~ao Paulo, SP, Brazil}
\date{\today}

\begin{abstract}
We prove the regularity of the $\eta$ function for classical pseudodifferential
operators with Shubin symbols. We recall the construction of complex
powers and of the Wodzicki and Kontsevich-Vishik functionals for classical
symbols on $\mathbb{R}^{n}$ with these symbols. We then define the
$\zeta$ and $\eta$ functions associated to suitable elliptic operators.
We compute the $K_{0}$ group of the algebra of zero-order operators
and use this knowledge to show that the Wodzicki trace of the idempotents
in the algebra vanishes. From this, it follows that the $\eta$ function
is regular at $0$ for any self-adjoint elliptic operator of positive
order.
Keywords: Pseudodifferential operators, K-Theory of $C^{*}$-algebras.
\end{abstract}

\maketitle

In this paper, we present a study of the $\eta$ and $\zeta$ functions
constructed for the class of classical elliptic symbols on $\mathbb{R}^{n}$
defined in Chapter 4 of Shubin \cite{Shubin}. These important
spectral functions can be associated with pseudodifferential
operators on a manifold, allowing the study of spectral properties
of the operators and carrying a good amount of geometric information.
The $\eta$ function was first defined by Atiyah, Patodi and Singer \cite{Atiyah75}. Its regularity at the origin was proved for particular cases one year later \cite{Atiyah76}, as part of their study of first order boundary problems. The complete proof of its
regularity was accomplished by Gilkey \cite{Gilkey}. Right after that,
Wodzicki gave an alternative analytical
proof in \cite{Wodzicki1,Wodzicki2}.

It is well known that the classes of Shubin symbols reproduce many of the properties of
the pseudodifferential calculus on compact manifolds. Furthermore many results
obtained recently about noncommutative aspects of global symbols, notably
those contained in the Nicola and Rodino's book \cite{NicolaRodino}, also
lead us to consider the natural question of the regularity of the
$\eta$ function for self-adjoint elliptic operators in this calculus.

We follow closely the ideas of Wodzicki and Shubin, explained clearly
also by Ponge \cite{Ponge}. It consists of the proof of the following
expression \[
\mbox{res}_{s=0}\eta(op(a),s)=2i\pi Res(\Pi_{\theta,\theta'}(a)),\]
where $\mbox{res}_{s=0}\eta(op(a),s)$ denotes the residue of the
meromorphic function $s\mapsto\eta(op(a),s)$, for a given self-adjoint
elliptic symbol of positive order $a$, $Res$ denotes the Wodzicki trace, see section 5.1 of \cite{NicolaRodino}, and $\Pi_{\theta,\theta'}(a)$
are the so-called {}``sectorial projections''. We can prove using only elementary $K$-Theory
of $C^{*}$-algebras that
the Wodzicki trace of any idempotent pseudodifferential operator in
the Shubin class is zero. Hence $\eta$ is a regular function at the
origin.

This article is organized as follows. We first recall the basic definitions and properties of the Shubin calculus and show how to determine the norm of zero order operators modulo compacts in terms of their principal symbols. We then define a Fr\'echet topology
for classical Shubin pseudodifferential operators and show that
the classical Shubin operators of zero order form a $*$-Fr\'echet algebra
contained in $\mathcal{B}(L^{2}(\mathbb{R}^{n}))$, the bounded operators
on $L^{2}(\mathbb{R}^{n})$, that is closed under the analytic functional
calculus.

After that we recall the definition of complex powers and the $\zeta$
function of Shubin pseudodifferential operators. The $\zeta$ function
provides a link between the $\eta$ function and the Wodzicki trace.
In order to define the $\zeta$ function we need to use complex
powers and a linear functional defined on the space of classical Shubin symbols with order different from $-2n+j$, $j\in\mathbb{N}_0$,
that extends in a certain way the trace of trace class pseudodifferential operators in
Hilbert spaces. We call this function, which appears also in the work
of Maniccia, Schrohe and Seiler \cite{DeterminantesSG} under the name
of ``finite-part integral'', the Kontsevich-Vishik functional
on classical Shubin operators, because it is similar to the function
defined in \cite{Vishik}.

We then prove that the Wodzicki trace of every pseudodifferential idempotent with symbol in the Shubin class 
is zero. In order to do that we compute the $K$-theory of the $C^{*}$-algebra
generated by the zero order operators and use standards arguments
of the theory. We use in particular that the $K$-theory of this $C^{*}$-algebra
is the same as the $K$-theory of the $*$-Fr\'echet algebra of the zero
order operators, because the latter is closed under the analytic functional
calculus. This implies the desired regularity of the $\eta$ function.

\section{Basic definitions.}

In this section we will recall some of the main properties of the
Shubin calculus. The basic facts about Shubin operators can be found
in details in the Chapter 4 of the book of Shubin \cite{Shubin} and
in Chapters 1 and 2 of the book of Nicola and Rodino \cite{NicolaRodino}.

In this paper, we fix a strictly positive $C^{\infty}$ function $[.]:\mathbb{R}^{n}\to\mathbb{R}$
such that $[x]=|x|$ for $|x|\ge1$. A zero excision function in $C^{\infty}(\mathbb{R}^n)$ is a function that assumes values on $[0,1]$ and which is zero in a neighborhood of the origin and it is equal to 1 outside a compact set. If $\mathcal{H}$ is a Hilbert
space, we use the notation $\mathcal{B}(\mathcal{H})$ to indicate
the set of bounded operators in $\mathcal{H}$. $\left\Vert .\right\Vert _{\mathcal{B}(\mathcal{H})}$
indicates the norm in this space. In particular $\mathcal{B}(\mathbb{C}^{q})$
coincides with the set of complex $q\times q$ matrices. We also use
$\mathbb{N}_{0}=\{0,1,2,...\}$ and $\mathbb{N}=\{1,2,3,...\}$.
\begin{defn}
\emph{(Shubin symbols)} Let $m\in\mathbb{C}$. We define $\Gamma^{m}(\mathbb{R}^{n},\mathcal{B}(\mathbb{C}^{q}))$
as the space of $C^{\infty}$ functions $a:\mathbb{R}^{n}\times\mathbb{R}^{n}\to\mathcal{B}(\mathbb{C}^{q})$,
such that for all $\alpha$ and $\beta\in\mathbb{N}_{0}^{n}$, there
exists a constant $C_{\alpha\beta}>0$ such that the following inequality
holds for all components $a_{ij}$ of $a$:\[
\left|\partial_{x}^{\beta}\partial_{\xi}^{\alpha}a_{ij}(x,\xi)\right|\le C_{\alpha\beta}[(x,\xi)]^{Re(m)-|\alpha|-|\beta|},\]
These estimates provide a Fréchet space structure to $\Gamma^{m}(\mathbb{R}^{n},\mathcal{B}(\mathbb{C}^{q}))$,
with seminorms\[
p_{\alpha,\beta}^{m}(a)=\sup_{(x,\xi)\in\mathbb{R}^{2n}}\left\Vert [(x,\xi)]^{-\left(Re(m)-|\alpha|-|\beta|\right)}\partial_{x}^{\beta}\partial_{\xi}^{\alpha}a(x,\xi)\right\Vert _{\mathcal{B}(\mathbb{C}^{q})}.\]
We note that $\Gamma^{-\infty}(\mathbb{R}^{n},\mathcal{B}(\mathbb{C}^{q}))=\cap_{m\in\mathbb{C}}\Gamma^{m}(\mathbb{R}^{n},\mathcal{B}(\mathbb{C}^{q}))=\mathcal{S}(\mathbb{R}^{2n},\mathcal{B}(\mathbb{C}^{q}))$,
the matrices whose entries are rapidly decreasing functions. The elements of $\Gamma^{-\infty}(\mathbb{R}^{n},\mathcal{B}(\mathbb{C}^{q}))$
are called regularizing symbols. When $q=1$, we denote $\Gamma^{m}(\mathbb{R}^{n},\mathcal{B}(\mathbb{C}^{q}))$
simply by $\Gamma^{m}(\mathbb{R}^{n})$ and call its elements of scalar symbols.
\end{defn}
Associated to each symbol $a\in\Gamma^{m}(\mathbb{R}^{n},\mathcal{B}(\mathbb{C}^{q}))$,
we can define an operator $op(a):\mathcal{S}(\mathbb{R}^{n})^{\oplus q}\to\mathcal{S}(\mathbb{R}^{n})^{\oplus q}$
as\[
\left(op(a)u(x)\right)_{l}=\frac{1}{\left(2\pi\right)^{n}}\sum_{j=1}^{q}\int e^{ix\xi}a_{lj}(x,\xi)\hat{u}_{j}(\xi)d\xi,\]
where $\hat{u}(\xi)=\int e^{-ix\xi}u(x)dx$. These operators extend to $op(a):\mathcal{S}(\mathbb{R}^{n})'^{\oplus q}\to\mathcal{S}(\mathbb{R}^{n})'^{\oplus q}$,
where $\mathcal{S}(\mathbb{R}^{n})'^{\oplus q}$ is the space of temperate
distributions. We note also that the map $a\mapsto op(a)$
is injective.

The operators of type $op(a)$ are closed under composition,
that is, if $a\in\Gamma^{m}(\mathbb{R}^{n},\mathcal{B}(\mathbb{C}^{q}))$
and $b\in\Gamma^{\mu}(\mathbb{R}^{n},\mathcal{B}(\mathbb{C}^{q}))$,
then $op(a)op(b)=op(c)$, where $c\in\Gamma^{m+\mu}(\mathbb{R}^{n},\mathcal{B}(\mathbb{C}^{q}))$
and it is also denoted by $a\sharp b$. The operators are also closed
under involution, that is, if $a\in\Gamma^{m}(\mathbb{R}^{n},\mathcal{B}(\mathbb{C}^{q}))$,
there is a unique $a^{*}\in\Gamma^{m}(\mathbb{R}^{n},\mathcal{B}(\mathbb{C}^{q}))$, such that for all $u$ and $v$ in $\mathcal{S}(\mathbb{R}^{n})$
the following holds: \[
\left(op(a)u,v\right)_{L^{2}(\mathbb{R}^{n})}=\left(u,op(a^{*})v\right)_{L^{2}(\mathbb{R}^{n})},\]
where $(,)_{L^{2}(\mathbb{R}^{n})}$ denotes the usual inner product of ${L^{2}(\mathbb{R}^{n})}$. In particular, if $a^{*}=a$, we say that $a$ is a self-adjoint symbol.

The subclasses of classical symbols are defined as follows:
\begin{defn}
We denote by $\Gamma^{(\mu)}(\mathbb{R}^{n},\mathcal{B}(\mathbb{C}^{q}))$,
$\mu\in\mathbb{C}$, the set of functions in $C^{\infty}\left((\mathbb{R}^{n}\times\mathbb{R}^{n})\backslash\{0\},\mathcal{B}(\mathbb{C}^{q})\right)$
that are homogeneous of order $\mu$, that is, $a_{(\mu)}(tx,t\xi)=t^{\mu}a_{(\mu)}(x,\xi)$
for all $t>0$. By $\Gamma_{cl}^{m}(\mathbb{R}^{n},\mathcal{B}(\mathbb{C}^{q}))$,
$m\in\mathbb{C}$, we denote the subset of $\Gamma^{m}(\mathbb{R}^{n},\mathcal{B}(\mathbb{C}^{q}))$
of symbols $a$, for which there exists a sequence of functions $a_{(m-j)}\in\Gamma^{(m-j)}(\mathbb{R}^{n},\mathcal{B}(\mathbb{C}^{q}))$,
$j\in\mathbb{N}_{0}$, such that, for any zero excision function $\chi\in C^{\infty}(\mathbb{R}^{n}\times\mathbb{R}^{n})$
and for all $N\in\mathbb{N}$, we have \[
a-\sum_{j=0}^{N-1}\chi a_{(m-j)}\in\Gamma^{m-N}(\mathbb{R}^{n},\mathcal{B}(\mathbb{C}^{q})).\]
In this case we say that $a$ is a classical symbol and the sequence $a_{(m-j)}$ is the asymptotic
expansion of $a$. We use the notation $a\sim\sum_{j=0}^{\infty}a_{(m-j)}$
to indicate that. The function $a_{(m)}\in\Gamma^{(m)}(\mathbb{R}^{n},\mathcal{B}(\mathbb{C}^{q}))$
is called principal symbol of $a$ if $a_{(m)}$ is not identically
equal to zero. The operators $op(a)$ with classical symbols - also
called classical operators - are closed under composition and involution.
\end{defn}
For classical symbols we define the following sequence \begin{equation}
0\to\Gamma_{cl}^{m-1}(\mathbb{R}^{n},\mathcal{B}(\mathbb{C}^{q}))\overset{i}{\to}\Gamma_{cl}^{m}(\mathbb{R}^{n},\mathcal{B}(\mathbb{C}^{q}))\overset{s}{\to}C^{\infty}(S^{2n-1},\mathcal{B}(\mathbb{C}^{q}))\to0,\label{eq:sequencia exata simbolo}\end{equation}
where $i$ is the inclusion and $s$ is given by $s(a)=a_{(m)}|_{S^{2n-1}}$.
This sequence is exact. In fact if $s(a)=0$, then $a_{(m)}=0$
and, by the asymptotic expansion, $a$ must belong to $\Gamma^{m-1}(\mathbb{R}^{n},\mathcal{B}(\mathbb{C}^{q}))$.
If $g\in C^{\infty}(S^{2n-1},\mathcal{B}(\mathbb{C}^{q}))$, we can choose a zero excision function
$\chi\in C^{\infty}(\mathbb{R}^{n}\times\mathbb{R}^{n})$ and define
$a(x,\xi)=\chi(x,\xi)g\left(\frac{(x,\xi)}{\left|(x,\xi)\right|}\right)\left|(x,\xi)\right|^{m}\in\Gamma_{cl}^{m}(\mathbb{R}^{n},\mathcal{B}(\mathbb{C}^{q}))$.
It is clear that $s(a)=g$, and $s$ is surjective.

Next we define the notion of elliptic symbols for this class. It will
turn out later that ellipticity is equivalent to the Fredholm property.
\begin{defn}
We say that $a\in\Gamma^{m}(\mathbb{R}^{n},\mathcal{B}(\mathbb{C}^{q}))$
is an elliptic symbol if there exists $C>0$ and $R>0$ such that
for each $|(x,\xi)|\ge R$, $a(x,\xi)$ is invertible and\[
\left\Vert a(x,\xi)^{-1}\right\Vert _{\mathcal{B}(\mathbb{C}^{q})}\le C[(x,\xi)]^{-Re(m)}.\]

\end{defn}
For classical symbols, ellipticity can be completely characterized
by the principal symbol: $a\in\Gamma_{cl}^{m}(\mathbb{R}^{n},\mathcal{B}(\mathbb{C}^{q}))$
is elliptic if, and only if, $a_{(m)}(x,\xi)$ is invertible for all
$(x,\xi)\in\mathbb{R}^{2n}\backslash\{0\}$, where $a_{(m)}$ is the
principal symbol of $a$. Moreover if $a\in\Gamma^{m}(\mathbb{R}^{n},\mathcal{B}(\mathbb{C}^{q}))$
is elliptic, then there exists a symbol $b\in\Gamma^{-m}(\mathbb{R}^{n},\mathcal{B}(\mathbb{C}^{q}))$,
called a parametrix of $a$, that satisfies $a\sharp b=1+r_{1}$ and
$b\sharp a=1+r_{2}$, where $r_{1}$ and $r_{2}$ belong to $\Gamma^{-\infty}(\mathbb{R}^{n},\mathcal{B}(\mathbb{C}^{q}))$. If $a$ is classical, so is its parametrix.

We can also define Sobolev type spaces, on which the operators act
continuously.
\begin{defn}
Let $b\in\Gamma^{s}(\mathbb{R}^{n})$ be the function $b(x,\xi)=\left(1+|x|^{2}+|\xi|^{2}\right)^{\frac{s}{2}}$,
$s\in\mathbb{R}$. Let $c\in\Gamma^{-s}(\mathbb{R}^{n})$ be a parametrix
of $b$. We define the Shubin Sobolev space $Q^{s}(\mathbb{R}^{n})$
as the space \[
Q^{s}(\mathbb{R}^{n})=\{u\in\mathcal{S}'(\mathbb{R}^{n});\, op(b)u\in L^{2}(\mathbb{R}^{n})\}.\]

Let $r\in\Gamma^{-\infty}(\mathbb{R}^{n})$ be such that $op(r)+I=op(c)op(b)$.
Then we can endow $Q^{s}(\mathbb{R}^{n})$ with a Hilbert space structure
with inner product\[
(u,v)_{s}=\left(op(b)u,op(b)v\right)_{L^{2}(\mathbb{R}^{n})}+\left(op(r)u,op(r)v\right)_{L^{2}(\mathbb{R}^{n})}.\]

\end{defn}
It can be proved that the spaces $Q^{s}(\mathbb{R}^{n})$ can be defined
using any other elliptic symbol and its parametrix instead of $b$.
That they provide the same topology is proved in Proposition 1.5.3
of \cite{NicolaRodino} for even more general operators.

These spaces are such that \[
\cap_{s\in\mathbb{R}}Q^{s}(\mathbb{R}^{n})=\mathcal{S}(\mathbb{R}^{n})\,\,\,\mbox{and}\,\,\,\cup_{s\in\mathbb{R}}Q^{s}(\mathbb{R}^{n})=\mathcal{S}'(\mathbb{R}^{n}).\]

Moreover if $a\in\Gamma^{m}(\mathbb{R}^{n},\mathcal{B}(\mathbb{C}^{q}))$,
then $op(a):Q^{s}(\mathbb{R}^{n})^{\oplus q}\to Q^{s-m}(\mathbb{R}^{n})^{\oplus q}$
acts continuously for all $s\in\mathbb{R}$. If $m=0$, this implies
that $op(a)\in\mathcal{B}(L^{2}(\mathbb{R}^{n})^{\oplus q})$. 

We note also that if $a\in\Gamma^{m}(\mathbb{R}^{n},\mathcal{B}(\mathbb{C}^{q}))$, $m>0$, is a self-adjoint elliptic symbol, then $op(a):Q^{m}(\mathbb{R}^{n})^{\oplus q}\subset L^{2}(\mathbb{R}^{n})^{\oplus q}\to L^{2}(\mathbb{R}^{n})^{\oplus q}$ is a unbounded self-adjoint operator. 

\section{Norm modulo compact operators.}

In this section, we express the norm of the operators, modulo compacts,
in terms of their principal symbols. The notation $\mathcal{K}(\mathcal{H})$
will indicate the set of compact operators on a Hilbert space $\mathcal{H}$.

We follow an idea of H\"ormander and define a family of unitary operators
on $\mathcal{B}(L^{2}(\mathbb{R}^{n}))^{\oplus q}$, which is our
adaptation for Shubin class of the family defined in Theorem 3.3 of
\cite{Hormander}.
\begin{defn}
\label{def:operadores para norma mod compacto} For each $\lambda>0$,
$(x_{0},\xi_{0})\in\mathbb{R}^{2n}\backslash\{0\}$, we define the
operator $T_{\lambda}(x_{0},\xi_{0}):L^{2}(\mathbb{R}^{n})^{\oplus q}\to L^{2}(\mathbb{R}^{n})^{\oplus q}$
by\[
T_{\lambda}(x_{0},\xi_{0})u(x)=e^{i\lambda^{\frac{1}{2}}(x-x_{0})\xi_{0}}u(x-\lambda^{\frac{1}{2}}x_{0}).\]

\end{defn}
These operators have the following properties:
\begin{prop}
\label{pro:prop de Rlambda, Slambda, Tlambda}Let $a\in\Gamma_{cl}^{0}(\mathbb{R}^{n},\mathcal{B}(\mathbb{C}^{q}))$
. Let us denote by $a_{(0)}$ the principal symbol of $a$. Let us
fix $(x_{0},\xi_{0})\in\mathbb{R}^{n}\times\mathbb{R}^{n}\backslash\{0\}$ and denote
by\textup{ $T_{\lambda}$} the operator\textup{ $T_{\lambda}(x_{0},\xi_{0})$}.
Then the following holds:

\textup{(i) $T_{\lambda}$ is an unitary operator acting on $L^{2}(\mathbb{R}^{n})^{\oplus q}$.}

\textup{(ii) For every $u\in L^{2}(\mathbb{R}^{n})^{\oplus q}$, we
have that $\lim_{\lambda\to\infty}T_{\lambda}u=0$ weakly.}

\textup{(iii) For any $u\in L^{2}(\mathbb{R}^{n})^{\oplus q}$ we
have \[
\lim_{\lambda\to\infty}\left\Vert T_{\lambda}^{-1}op(a)T_{\lambda}u-a_{(0)}(x_{0},\xi_{0})u\right\Vert _{L^{2}(\mathbb{R}^{n})^{\oplus q}}=0.\]
}\end{prop}
\begin{proof}
We only sketch the proof since it follows closely the arguments of
Theorem 3.3 of H\"ormander \cite{Hormander} and the arguments that
follow lemma 4.6 of Grigis and Sj\"ostrand \cite{Sjoestrand}.
We just note that\[
\widehat{T_{\lambda}u}(\xi)=e^{i\lambda x_{0}\xi_{0}-i\lambda^{\frac{1}{2}}x_{0}\xi_{0}-i\lambda^{\frac{1}{2}}x_{0}\xi}\hat{u}(\xi-\lambda^{\frac{1}{2}}\xi_{0})\]
\[
T_{\lambda}^{-1}u(x)=e^{-i\lambda^{\frac{1}{2}}(x+\lambda^{\frac{1}{2}}x_{0}-x_{0})\xi_{0}}u(x+\lambda^{\frac{1}{2}}x_{0}).\]

Using these two formulas we obtain the following\[
T_{\lambda}^{-1}op(a)T_{\lambda}u=op(a(x+\lambda^{\frac{1}{2}}x_{0},\xi+\lambda^{\frac{1}{2}}\xi_{0})).\]

The proof follows now from the Lebesgue dominated convergence Theorem.\end{proof}
\begin{cor}
Let $A=op(a)$, $a\in\Gamma_{cl}^{0}(\mathbb{R}^{n})$. Then\[
\sup_{|(x,\xi)|=1}\left|a_{(0)}(x,\xi)\right|\le\inf_{C\in\mathcal{K}(L^{2}(\mathbb{R}^{n}))}\left\Vert A+C\right\Vert _{\mathcal{B}(L^{2}(\mathbb{R}^{n}))}.\]
\end{cor}
\begin{proof}
Let us prove that $\left|a_{(0)}(x_{0},\xi_{0})\right|\le\inf_{C\in\mathcal{K}(L^{2}(\mathbb{R}^{n}))}\left\Vert A+C\right\Vert _{\mathcal{B}(L^{2}(\mathbb{R}^{n}))}$
for any $(x_{0},\xi_{0})\ne 0$.

For any $\epsilon>0$, $C\in\mathcal{K}(L^2(\mathbb{R}^n))$ and $u\in L^{2}(\mathbb{R}^{n})$,
there is a $\lambda>0$ such that \[
\left\Vert T_{\lambda}^{-1}(A+C)T_{\lambda}u-a_{(0)}(x_{0},\xi_{0})u\right\Vert _{L^{2}(\mathbb{R}^{n})}<\epsilon.\]
Therefore $\left|a_{(0)}(x_{0},\xi_{0})\right|\le\left\Vert A+C\right\Vert _{\mathcal{B}(L^{2}(\mathbb{R}^{n}))}+\epsilon$
for any $\epsilon>0$ and then the result follows.
\end{proof}
We will prove next the opposite inequality. A similar argument was used by Nicola for SG symbols \cite{Nicolakteoria}. In order to do that we use simple facts about the Anti-Wick quantization, see \cite[Section 1.7]{NicolaRodino} and \cite[Chapter 4]{Shubin}. These are stated in the form of the next proposition.
\begin{prop}
There is a linear map $a\in\Gamma_{cl}^{0}(\mathbb{R}^{n})\mapsto A_{a}\in\mathcal{B}(L^{2}(\mathbb{R}^{n}))$
such that the following holds:

(i) $A_{1}=I$, where $I$ is the identity operator on $\mathcal{B}(L^{2}(\mathbb{R}^{n}))$.

(ii) For each $a\in\Gamma_{cl}^{0}(\mathbb{R}^{n})$, there exists
$b\in\Gamma_{cl}^{0}(\mathbb{R}^{n})$ such that $A_{a}=op(b)$ and
$a-b\in\Gamma_{cl}^{-1}(\mathbb{R}^{n})$.

(iii) If $a\in\Gamma_{cl}^{0}(\mathbb{R}^{n})$ is a real valued function,
then $A_{a}$ is a self-adjoint operator. Furthermore, if $a\ge0$, then $A_{a}\ge0$.
\end{prop}
Using this we can prove.
\begin{prop}
Let $A=op(a)$, $a\in\Gamma_{cl}^{0}(\mathbb{R}^{n})$ and $a_{(0)}$
be its principal symbol. Then \[
\inf_{C\in\mathcal{K}(L^{2}(\mathbb{R}^{n}))}\left\Vert A+C\right\Vert _{\mathcal{B}(L^{2}(\mathbb{R}^{n}))}\le2\sup_{|(x,\xi)|=1}\left|a_{(0)}(x,\xi)\right|.\]
\end{prop}
\begin{proof}
It is enough to prove that for any real valued function $a\in\Gamma_{cl}^{0}(\mathbb{R}^{n})$,
we have\[
\inf_{C\in\mathcal{K}(L^{2}(\mathbb{R}^{n}))}\left\Vert A+C\right\Vert _{\mathcal{B}(L^{2}(\mathbb{R}^{n}))}\le\sup_{|(x,\xi)|=1}\left|a_{(0)}(x,\xi)\right|.\]
If $a$ is a real valued function, so is its principal symbol $a_{(0)}$. Let $M:=\sup_{|(x,\xi)|=1}\left|a_{(0)}(x,\xi)\right|$ and $\chi\in C^{\infty}(\mathbb{R}^{n}\times\mathbb{R}^{n})$ be
a zero excision function. Let us define $c\in\Gamma_{cl}^{0}(\mathbb{R}^{n})$
by\[
c(x,\xi)=\chi(x,\xi)a_{(0)}(x,\xi).\]

The above expression shows that the principal symbol of $c$, $c_{(0)}$,
is equal to $a_{(0)}$. Therefore $a-c\in\Gamma_{cl}^{-1}(\mathbb{R}^{n})$.
It is also clear that $M-c\ge0$ and $M+c\ge0$. Hence using the Anti-Wick
quantization, we conclude that $MI-A_{c}$ and $MI+A_{c}$ are
two positive operators on $\mathcal{B}(L^{2}(\mathbb{R}^{n}))$. As $A_{c}$ is self-adjoint, $\left\Vert A_{c}\right\Vert _{\mathcal{B}(L^{2}(\mathbb{R}^{n})}\le M$.
We know that $A_{c}=op(b)$, for some $b\in\Gamma_{cl}^{0}(\mathbb{R}^{n})$
such that $c-b\in\Gamma_{cl}^{-1}(\mathbb{R}^{n})$. Hence $A_{c}=op(a)+op(k)$,
where $k=b-c+c-a\in\Gamma_{cl}^{-1}(\mathbb{R}^{n})$. Therefore $op(k)\in\mathcal{K}(L^{2}(\mathbb{R}^{n}))$
and $\left\Vert op(a)+op(k)\right\Vert _{\mathcal{B}(L^{2}(\mathbb{R}^{n}))}\le M$.
\end{proof}
Let $\overline{\Gamma_{cl}^{0}(\mathbb{R}^{n})}$ be the closure of
the set of operators $op(a)$ in $\mathcal{B}(L^{2}(\mathbb{R}^{n}))$,
with $a\in\Gamma_{cl}^{0}(\mathbb{R}^{n})$. As a consequence of the
above proposition one can define the operator $\overline{s}:\overline{\Gamma_{cl}^{0}(\mathbb{R}^{n})}\to C(S^{2n-1})$
as the unique operator such that \[
\overline{s}(op(a))=s(a),\]
where $s$ is the function that was defined in \ref{eq:sequencia exata simbolo}.

In fact as $\sup_{|(x,\xi)|=1}\left|a_{(0)}(x,\xi)\right|\le\left\Vert A\right\Vert _{\mathcal{B}(L^{2}(\mathbb{R}^{n}))}$,
the above function extends continuously to $\overline{\Gamma_{cl}^{0}(\mathbb{R}^{n})}$.
\begin{cor}
\label{thm:norma modulo compactos} The following sequence\begin{equation}
0\to\mathcal{K}(L^{2}(\mathbb{R}^{n}))\overset{\overline{i}}{\to}\overline{\Gamma_{cl}^{0}(\mathbb{R}^{n})}\overset{\overline{s}}{\to}C(S^{2n-1})\to0,\label{eq:sequencia exata}\end{equation}
where $\overline{i}$ is the inclusion and $\overline{s}$ is
the map defined above, is an exact sequence of $C^{*}$-algebras.
In particular $\overline{\Gamma_{cl}^{0}(\mathbb{R}^{n})}/\mathcal{K}$
is isomorphic to $C(S^{2n-1})$, which implies that for $A=op(a)$,
$a\in\Gamma_{cl}^{0}(\mathbb{R}^{n})$, the following equality holds\[
\inf_{C\in\mathcal{K}(L^{2}(\mathbb{R}^{n}))}\left\Vert A+C\right\Vert _{\mathfrak{\mathcal{B}}(L^{2}(\mathbb{R}^{n}))}=\sup_{|(x,\xi)|=1}\left|a_{(0)}(x,\xi)\right|.\]
\end{cor}
\begin{proof}
It is clear that the inclusion $\overline{i}$ is injective and the
inequalities we have just proved show that $\ker\overline{s}=\mbox{Im}(\overline{i})$.
In order to prove that $\overline{s}$ is surjective, we 
have only to observe that $s$ is surjective and $C^{\infty}(S^{2n-1})$
is dense in $C(S^{2n-1})$. The result now follows from the fact that
the image of an homomorphism of $C^{*}$ algebras is closed.
\end{proof}

\section{Fréchet spaces and algebras of classical Shubin operators.}

In order to deal with classical Shubin symbols, we will consider
the following topology in $\Gamma_{cl}^{m}(\mathbb{R}^{n},\mathcal{B}(\mathbb{C}^{q}))$.
\begin{cor}
The topology of $\Gamma_{cl}^{m}(\mathbb{R}^{n},\mathcal{B}(\mathbb{C}^{q}))$,
$m\in\mathbb{C}$, is the smallest locally convex topology that makes
the following maps continuous\[
a\in\Gamma_{cl}^{m}(\mathbb{R}^{n},\mathcal{B}(\mathbb{C}^{q}))\to a_{(m-j)}\in\Gamma^{(m-j)}(\mathbb{R}^{n},\mathcal{B}(\mathbb{C}^{q})),\]
\[
a\in\Gamma_{cl}^{m}(\mathbb{R}^{n},\mathcal{B}(\mathbb{C}^{q}))\to a-\sum_{l=0}^{j-1}\chi a_{(m-l)}(x,\xi)\in\Gamma^{m-j}(\mathbb{R}^{n},\mathcal{B}(\mathbb{C}^{q})),\]
where each $j\in\mathbb{N}_{0}$ defines a different map. In the above
expression $\chi$ is a zero excision function, $a_{(m-j)}$ are the
homogeneous terms of the asymptotic expansion of $a$ and $\Gamma^{(\mu)}(\mathbb{R}^{n},\mathcal{B}(\mathbb{C}^{q}))$
is the Fréchet space, whose topology is induced by
$C^{\infty}\left(\left(\mathbb{R}^{n}\times\mathbb{R}^{n}\right)\backslash\{0\},\mathcal{B}(\mathbb{C}^{q})\right)$.
\end{cor}
The above topology is clearly a Fréchet one and it is independent
of the zero excision function $\chi$. It is clear that the inclusion
$\Gamma_{cl}^{m}(\mathbb{R}^{n},\mathcal{B}(\mathbb{C}^{q}))\hookrightarrow\Gamma^{m}(\mathbb{R}^{n},\mathcal{B}(\mathbb{C}^{q}))$
is continuous. Furthermore the space $\Gamma^{-\infty}(\mathbb{R}^{n},\mathcal{B}(\mathbb{C}^{q}))=\mathcal{S}(\mathbb{R}^{2n},\mathcal{B}(\mathbb{C}^{q}))$
is closed in $\Gamma_{cl}^{m}(\mathbb{R}^{n},\mathcal{B}(\mathbb{C}^{q}))$,
for any $m\in\mathbb{C}$. Actually $\Gamma_{cl}^{m}(\mathbb{R}^{n},\mathcal{B}(\mathbb{C}^{q}))$
induces in $\Gamma^{-\infty}(\mathbb{R}^{n},\mathcal{B}(\mathbb{C}^{q}))$
the usual topology of $\mathcal{S}(\mathbb{R}^{n}\times\mathbb{R}^{n},\mathcal{B}(\mathbb{C}^{q}))$.

This topology is such that the involution $*:\Gamma_{cl}^{m}(\mathbb{R}^{n},\mathcal{B}(\mathbb{C}^{q}))\to\Gamma_{cl}^{m}(\mathbb{R}^{n},\mathcal{B}(\mathbb{C}^{q}))$
and the operation of composition $\sharp:\Gamma_{cl}^{m}(\mathbb{R}^{n},\mathcal{B}(\mathbb{C}^{q}))\times\Gamma_{cl}^{\mu}(\mathbb{R}^{n},\mathcal{B}(\mathbb{C}^{q}))\to\Gamma_{cl}^{m+\mu}(\mathbb{R}^{n},\mathcal{B}(\mathbb{C}^{q}))$
are continuous. In particular $\Gamma_{cl}^{0}(\mathbb{R}^{n},\mathcal{B}(\mathbb{C}^{q}))$
is a $*$-algebra of Fréchet.
\begin{defn}
Let $op:\Gamma_{cl}^{0}(\mathbb{R}^{n},\mathcal{B}(\mathbb{C}^{q}))\to\mathcal{B}(L^{2}(\mathbb{R}^{n})^{\oplus q})$
be the continuous injective function given by $a\mapsto op(a)$. We
define $\mathcal{A}$ as the $*$-Fréchet algebra $op(\Gamma_{cl}^{0}(\mathbb{R}^{n},\mathcal{B}(\mathbb{C}^{q})))$
with the topology induced by $\Gamma_{cl}^{0}(\mathbb{R}^{n},\mathcal{B}(\mathbb{C}^{q}))$. 
\end{defn}
Next we prove that $\mathcal{A}$ is a $\Psi^{*}$-algebra in the
sense of Gramsch \cite{Gramsch}.
\begin{prop}
Let $a\in\Gamma_{cl}^{\mu}(\mathbb{R}^{n},\mathcal{B}(\mathbb{C}^{q}))$,
$\mu\in\mathbb{C}$. Hence $a$ is elliptic iff $op(a):\mathcal{Q}^{s}(\mathbb{R}^{n})^{\oplus q}\to\mathcal{Q}^{s-Re(\mu)}(\mathbb{R}^{n})^{\oplus q}$
is Fredholm for every $s\in\mathbb{R}$.
\end{prop}
\begin{proof}
We prove only for $a\in\Gamma_{cl}^{0}(\mathbb{R}^{n},\mathcal{B}(\mathbb{C}^{q}))$.
The general result follows by the use of order reducing operators.
By a order reducing operator of order $m$, we mean an operator $op(a)$,
where $a\in\Gamma_{cl}^{m}(\mathbb{R}^{n},\mathcal{B}(\mathbb{C}^{q}))$
is an elliptic invertible symbol, that is, there exists $b\in\Gamma_{cl}^{m}(\mathbb{R}^{n},\mathcal{B}(\mathbb{C}^{q}))$ such that $a\sharp b=b\sharp a=1$. Hence $op(a)$ induces an isomorphism from $Q^{s}(\mathbb{R}^{n})^{\oplus q}$
to $Q^{s-m}(\mathbb{R}^{n})^{\oplus q}$ for any $s$. The construction
of these operators for the Shubin symbols can be done by the use of
the Anti-Wick quantization, as it is explained in section 1.7 of \cite{NicolaRodino}.

If $a$ is elliptic, then there exists $b\in\Gamma_{cl}^{0}(\mathbb{R}^{n},\mathcal{B}(\mathbb{C}^{q}))$
such that $a\sharp b\equiv b\sharp a\equiv1\,\mbox{mod}\Gamma^{-\infty}(\mathbb{R}^{n},\mathcal{B}(\mathbb{C}^{q}))$.
Hence $op(a)op(b)=1+C_{1}$ and $op(b)op(a)=1+C_{2}$ where $C_{1}$
and $C_{2}$ are compact operators in $Q^s(\mathbb{R}^{n})^{\oplus q}$
for any $s$. We conclude that $op(a)$
is a Fredholm operator.

Let us suppose that $A=op(a):L^2(\mathbb{R}^n)^{\oplus q}\to L^2(\mathbb{R}^n)^{\oplus q}$ is Fredholm. Then there exists an operator
$B\in\mathcal{B}(L^{2}(\mathbb{R}^{n})^{\oplus q})$ such that $BA=I+C$,
$C\in \mathcal{K}(L^2(\mathbb{R}^n)^{\oplus q})$. Let us now fix $(x_{0},\xi_{0})\in\mathbb{R}^{2n}\backslash(0,0)$
and let us denote the operators $T_{\lambda}(x_{0},\xi_{0})$, $\lambda>0$,
just by $T_{\lambda}$ (see definition \ref{def:operadores para norma mod compacto}).
Then we have\[
0<\left\Vert u\right\Vert _{L^{2}(\mathbb{R}^{n})^{\oplus q}}=\left\Vert T_{\lambda}u\right\Vert _{L^{2}(\mathbb{R}^{n})^{\oplus q}}=\left\Vert \left(BA-C\right)T_{\lambda}u\right\Vert _{L^{2}(\mathbb{R}^{n})^{\oplus q}}\le\left\Vert BT_{\lambda}T_{\lambda}^{-1}AT_{\lambda}u\right\Vert _{L^{2}(\mathbb{R}^{n})^{\oplus q}}+\left\Vert CT_{\lambda}u\right\Vert _{L^{2}(\mathbb{R}^{n})^{\oplus q}}\le\]
\[
\le\left\Vert B\right\Vert _{\mathcal{B}(L^{2}(\mathbb{R}^{n})^{\oplus q})}\left\Vert T_{\lambda}\right\Vert _{\mathcal{B}(L^{2}(\mathbb{R}^{n})^{\oplus q})}\left\Vert T_{\lambda}^{-1}AT_{\lambda}u\right\Vert _{L^{2}(\mathbb{R}^{n})^{\oplus q}}+\left\Vert CT_{\lambda}u\right\Vert _{L^{2}(\mathbb{R}^{n})^{\oplus q}}\to\left\Vert B\right\Vert _{\mathcal{B}(L^{2}(\mathbb{R}^{n})^{\oplus q})}\left\Vert a_{(0)}(x_{0},\xi_{0})u\right\Vert _{L^{2}(\mathbb{R}^{n})^{\oplus q}}.\]

We conclude that $\left\Vert a_{(0)}(x_{0},\xi_{0})\right\Vert _{\mathcal{B}(\mathbb{C}^{q})}\ge\left\Vert B\right\Vert _{\mathcal{B}(L^{2}(\mathbb{R}^{n})^{\oplus q})}^{-1}$
for all $(x_{0},\xi_{0})\ne(0,0)$. Hence $a$ is elliptic.
\end{proof}
If $A=op(a):L^{2}(\mathbb{R}^{n},\mathcal{B}(\mathbb{C}^{q}))\to L^{2}(\mathbb{R}^{n},\mathcal{B}(\mathbb{C}^{q}))$,
$a\in\Gamma_{cl}^{0}(\mathbb{R}^{n},\mathcal{B}(\mathbb{C}^{q}))$,
is invertible as an element of $\mathcal{B}(L^{2}(\mathbb{R}^{n}))$,
then $A$ is Fredholm and $a$ is elliptic. Therefore there exists
$b\in\Gamma_{cl}^{0}(\mathbb{R}^{n},\mathcal{B}(\mathbb{C}^{q}))$
such that $op(b)op(a)=I+R_{1}$ and $op(a)op(b)=I+R_{2}$, where $R_{1}$
and $R_{2}$ are regularizing. Using these relations it is easy to
get\[
A^{-1}=op(b)-op(b)R_{2}+R_{1}A^{-1}R_{2}\in\Gamma_{cl}^{0}(\mathbb{R}^{n},\mathcal{B}(\mathbb{C}^{q})).\]
Therefore $A^{-1}=op(c)$, for some $c\in\Gamma_{cl}^{0}(\mathbb{R}^{n},\mathcal{B}(\mathbb{C}^{q}))$.
We obtain the following corollary.
\begin{cor}
$\mathcal{A}$ is a $*$\textup{-Fréchet algebra that is spectrally
invariant in $\mathfrak{A}$, that is, }$\mathcal{A}\cap\mathcal{\mathfrak{A}}^{-1}=\mathcal{A}^{-1}$,\textup{
and it has a stronger topology then that induced by $\mathfrak{A}$.
It is therefore what Gramsch called an $\Psi^{*}$-algebra.}
\end{cor}
The group of invertible elements of a $\Psi^{*}$-algebra is always
open. In any unital Fréchet algebra with an open group of invertible
elements, the inversion is continuous (see Waelbroeck \cite{Waelbroeck}).
We then conclude that the inversion in the algebras $\mathcal{A}$
and $\Gamma_{cl}^{0}(\mathbb{R}^{n},\mathcal{B}(\mathbb{C}^{q}))$
is continuous.

Now let us consider $a\in\Gamma_{cl}^{0}(\mathbb{R}^{n},\mathcal{B}(\mathbb{C}^{q}))$.
Due to the spectral invariance we know that $\sigma(a)=\sigma(op(a))$,
where $\sigma$ denotes the spectrum of $a$ and $op(a)$ in $\mathcal{A}$
and $\mathfrak{A}$ respectively. Note that we are identifying the
two algebras $\mathcal{A}$ and $\Gamma_{cl}^{0}(\mathbb{R}^{n},\mathcal{B}(\mathbb{C}^{q}))$.
Let $f$ be an analytic function in $\Omega$, an open set that contains
$\sigma(a)$. Next we show that $f(op(a))$, the operator constructed
using the holomorphic functional calculus, belongs to $\mathcal{A}$.

Let $\lambda\in\mathbb{C}\backslash\sigma(a)$. Then $\lambda-a$
has an inverse. Therefore $(\lambda-a)^{-1}=b(\lambda)\in\Gamma_{cl}^{0}(\mathbb{R}^{n},\mathcal{B}(\mathbb{C}^{q}))$
and $(\lambda-op(a))^{-1}=op(b(\lambda))$. As $\lambda\in\Omega\backslash\sigma(a)\mapsto\lambda-a\in\Gamma_{cl}^{0}(\mathbb{R}^{n},\mathcal{B}(\mathbb{C}^{q}))$
is continuous, we conclude that $\lambda\in\Omega\backslash\sigma(a)\mapsto b(\lambda)\in\Gamma_{cl}^{0}(\mathbb{R}^{n},\mathcal{B}(\mathbb{C}^{q}))$
is also continuous.

Now if $\gamma$ is a curve around $\sigma(a)$ in the positive direction
and contained in $\Omega$, we conclude from the continuity of the
inversion that the function $f(a)$ defined by the expression $\frac{1}{2\pi i}\int_{\gamma}f(\lambda)b(x,\xi,\lambda)d\lambda$
belongs to $\Gamma_{cl}^{0}(\mathbb{R}^{n},\mathcal{B}(\mathbb{C}^{q}))$.
Moreover if $b\sim\sum_{j}b_{(-j)}$, the terms of the asymptotic
expansion of $f(a)$ are\[
\frac{1}{2\pi i}\int_{\gamma}f(\lambda)b_{(-j)}(x,\xi,\lambda)d\lambda\in\Gamma^{(-j)}(\mathbb{R}^{n},\mathcal{B}(\mathbb{C}^{q})).\]

This implies that under the hypothesis above, the operator $f(op(a))$
constructed using the usual holomorphic functional calculus in $\mathcal{B}(L^{2}(\mathbb{R}^{n})^{\oplus q})$
is a pseudodifferential operator, whose symbol is $\frac{1}{2\pi i}\int_{\gamma}f(\lambda)b(x,\xi,\lambda)d\lambda$. Similar arguments to these were also used in a more difficult context in \cite{Schrohehinfty}.

\section{Complex powers of Shubin operators.}

In this section, we recall the construction of complex powers and
sectorial projection. The details can be found for a similar context,
the SG calculus, in the article of Mannicia, Schrohe and Seiler \cite{Schrohecomplex}, see also Robert \cite{Robert} and Kumano-go \cite{Kumanogo}.
For Shubin symbols the constructions are the same, even simpler actually, and can be found in \cite{Tese}.

\subsection{The construction of complex powers.}

Let $\Lambda$ be a sector of $\mathbb{C}$, $\Lambda=\{re^{i\varphi}\in\mathbb{C};\,0\le r<\infty,\,\theta\le\varphi\le\theta'\}$,
also denoted by $\Lambda_{\theta,\theta'}$. We will first define
the concept of $\Lambda$-ellipticity.
\begin{defn}
A classical symbol $a\in\Gamma_{cl}^{m}(\mathbb{R}^{n},\mathcal{B}(\mathbb{C}^{q}))$,
$m\in\mathbb{R}$ and $m\ge0$, is $\Lambda$-elliptic iff \[
spec(a_{(m)}(x,\xi))\cap\Lambda=\emptyset,\,\,\,\forall|(x,\xi)|\ne0,\]
 where $a_{(m)}$ is the principal symbol of $a$ and $spec$ is the
spectrum of a matrix in $\mathcal{B}(\mathbb{C}^{q})$. In particular,
$a$ is elliptic.
\end{defn}

Using the same analysis done by Mannicia, Schrohe and Seiler \cite{Schrohecomplex},
we can construct a parametrix $b(\lambda)$ of $(\lambda-a)$ in the following way. We define\[
b_{-m}(x,\xi,\lambda):=(\lambda-a_{(m)}(x,\xi))^{-1},\]
and
\[
b_{-m-k}(x,\xi):=\sum_{j+2|\sigma|+p=k,\, j<k}\frac{1}{\alpha!}(\partial_{\xi}^{\sigma}b_{-m-j})(x,\xi,\lambda)(D_{x}^{\sigma}a_{(m-p)})(x,\xi)b_{-m}(x,\xi,\lambda).\]

Using standards techniques of asymptotic summation, one can construct a parametrix $b(\lambda)$ such that for any $N\ge1$ and for any zero excision function $\chi$, there
is a constant $C>0$ such that \[
\left\Vert \partial_{\xi}^{\alpha}\partial_{x}^{\beta}\left(b(x,\xi,\lambda)-\sum_{k=0}^{N-1}\chi(x,\xi)b_{-m-k}(x,\xi,\lambda)\right)\right\Vert _{\mathcal{B}(\mathbb{C}^{q})}\le C(|\lambda|+[(x,\xi)]^{m})^{-3}[(x,\xi)]^{2m-N-|\alpha|-|\beta|}\]
and\[
\left\Vert \partial_{\xi}^{\alpha}\partial_{x}^{\beta}b(x,\xi,\lambda)\right\Vert _{\mathcal{B}(\mathbb{C}^{q})}\le C(|\lambda|+[(x,\xi)]^{m})^{-1}[(x,\xi)]^{-|\alpha|-|\beta|}.\]

This
parametrix $b$ is a $C^{\infty}$ function of $(x,\xi,\lambda)\in\mathbb{R}^{2n}\times\Lambda$,
such that for each $\lambda$, $b(x,\xi,\lambda)$ belongs to $\Gamma_{cl}^{-m}(\mathbb{R}^{n},\mathcal{B}(\mathbb{C}^{q}))$.
Moreover for any continuous seminorm $q$ of $\Gamma^{-m}(\mathbb{R}^{n},\mathcal{B}(\mathbb{C}^{q}))$,
 $[\lambda]q\left(b(\lambda)\right)$ are bounded for
$\lambda\in\Lambda$ and for any continuous seminorm $\tilde{q}$ of $\Gamma^{-\infty}(\mathbb{R}^{n},\mathcal{B}(\mathbb{C}^{q}))$,
 $[\lambda]^2\tilde{q}\left((\lambda-a)^{-\sharp}-b(\lambda))\right)$ are bounded for
$\lambda\in\Lambda$, where $^{-\sharp}$ indicates the inversion in the symbolic calculus. This parametrix is such that\[
(\lambda-a)\sharp b(\lambda)=1+r_{1}(\lambda),\]
\[
b(\lambda)\sharp(\lambda-a)=1+r_{2}(\lambda).\]

In the above expression $r_{i}(\lambda)\in\Gamma^{-\infty}(\mathbb{R}^{n},\mathcal{B}(\mathbb{C}^{q}))$
for $i=1$ and $2$. For any continuous seminorm $q$ of $\Gamma^{-\infty}(\mathbb{R}^{n},\mathcal{B}(\mathbb{C}^{q}))$,
that is, of $\mathcal{S}(\mathbb{R}^{n}\times\mathbb{R}^{n},\mathcal{B}(\mathbb{C}^{q}))$,
we have that $[\lambda]q\left(r_{i}(\lambda)\right)$ are bounded
for $\lambda\in\Lambda$. In particular the resolvent set of the unbounded
operator\[
A:=op(a):\mathcal{Q}^{m}(\mathbb{R}^{n})^{\oplus q}\subset L^{2}(\mathbb{R}^{n})^{\oplus q}\to L^{2}(\mathbb{R}^{n})^{\oplus q}\]
 contains all $\lambda\in\Lambda$ of sufficiently large absolute
value.

Our operators must satisfy conditions stronger than only $\Lambda$-ellipticity.
In fact we will consider only operators $A=op(a)$ such that 
\begin{enumerate}
\item $a\in\Gamma_{cl}^{m}(\mathbb{R}^{n},\mathcal{B}(\mathbb{C}^{q}))$
is $\Lambda$-elliptic, with $m>0$.
\item $\lambda-op(a)$ is invertible for all $0\ne\lambda\in\Lambda$.
\end{enumerate}
We call these conditions as assumption $(A)$. The spectrum of $op(a):\mathcal{Q}^{m}(\mathbb{R}^{n})^{\oplus q}\subset L^{2}(\mathbb{R}^{n})^{\oplus q}\to L^{2}(\mathbb{R}^{n})^{\oplus q}$,
where $a\in\Gamma_{cl}^{m}(\mathbb{R}^{n},\mathcal{B}(\mathbb{C}^{q}))$
is elliptic and $m>0$, is always either equal to $\mathbb{C}$ or
pure point without acummulation point \cite[Proposition 4.2.5]{NicolaRodino}. Hence if $a$ is $\Lambda$-elliptic,
the spectrum has to be pure point and we can always choose a subsector
$\Lambda'\subset\Lambda$, such that the second condition is satisfied.
We note also that in this case $0$ is at most an isolated spectral
point.

We are now in condition to define complex powers and sectorial projections. Let $\Lambda$ and $\tilde{\Lambda}$ be two different
sectors of $\mathbb{C}$. Suppose that $a$ is $\Lambda$ and $\tilde{\Lambda}$
elliptic. Let us define the following curves. In the following $\theta$ and $\theta'$ are angles whose rays $\{re^{i\theta},0\le r<\infty\}$ and $\{re^{i\theta'},0\le r<\infty\}$ are
contained on $\Lambda$ and $\tilde{\Lambda}$, respectively.

\[
\Gamma_{\theta,\theta'}=\{\rho e^{i\theta},\infty>\rho\ge\epsilon\}\cup\{\epsilon e^{it};\theta\le t\le\theta'\}\cup\{\rho e^{i\theta'},\epsilon\le\rho<\infty\}\]

and \[
\Gamma_{\theta}=\{\rho e^{i(\theta-2\pi)},\infty>\rho\ge\epsilon\}\cup\{\epsilon e^{it};\theta-2\pi\le t\le\theta\}\cup\{\rho e^{i\theta},\epsilon\le\rho<\infty\},\]
where $\epsilon>0$ is such that $\lambda-op(a)$ is invertible for all $\lambda \in B_{\epsilon}(0)\backslash\{0\}$.

From now on we will denote by $\gamma$ any one of these curves. The fact that the inverse of $\lambda-a$ can be described by the above parameter dependent parametrix implies the following:
\begin{prop}
\label{pro:integral e simbolo}Let $a\in\Gamma_{cl}^{m}(\mathbb{R}^{n},\mathcal{B}(\mathbb{C}^{q}))$
be a symbol that satisfies assumption $(A)$. Then $\int_{\gamma}\lambda^{z}(a-\lambda)^{-\sharp}d\lambda\in\Gamma_{cl}^{mz}(\mathbb{R}^{n},\mathcal{B}(\mathbb{C}^{q}))$
for any $z\in\mathbb{C}$ such that $Re(z)<0$.\end{prop}
\begin{cor}
Under the above hypothesis and for  $Re(z)<0$, the bounded operator on $L^2(\mathbb{R}^n)^{\oplus q}$ given by the integral \[
F(A)=\int_{\gamma}\lambda^{z}(op(a)-\lambda)^{-1}d\lambda.\] is a pseudodifferential operator with symbol $\int_{\gamma}\lambda^{z}(a-\lambda)^{-\sharp}d\lambda\in\Gamma_{cl}^{mz}(\mathbb{R}^{n},\mathcal{B}(\mathbb{C}^{q}))$.

\end{cor}
We are now in conditions to define the complex powers of a pseudodifferential
operator and the sectorial projection.
\begin{defn}
\label{def:complexpowers}Let $a\in\Gamma_{cl}^{m}(\mathbb{R}^{n},\mathcal{B}(\mathbb{C}^{q}))$,
$m>0$, be a classical $\Lambda$-elliptic pseudodifferential operator
that satisfies the assumption $(A)$ and $\theta$ be an angle whose
ray $\{re^{i\theta};0\le r<\infty\}$ is contained in $\Lambda$.
Let us define $\lambda_{\theta}^{z}=|\lambda|^{z}e^{iz\mbox{ar\ensuremath{g_{\theta}}}\lambda}$,
for $\lambda\in\mathbb{C}\backslash\{te^{i\theta};t\ge0\}$, where
$\mbox{ar\ensuremath{g_{\theta}}}\lambda$ assumes values in $]\theta-2\pi,\theta[$.
The complex powers of $a$ are the pseudodifferential operators, whose
symbols are denoted by $a_{\theta}^{z}\in\Gamma_{cl}^{mz}(\mathbb{R}^{n},\mathcal{B}(\mathbb{C}^{q}))$,
given by\[
op(a_{\theta}^{z})=\frac{1}{2\pi i}\int_{\Gamma_{\theta}}\lambda_{\theta}^{z}(\lambda-op(a))^{-1}d\lambda,\,\,\, Re(z)<0.\]

For $Re(z)>0$, we define $op(a_{\theta}^{z}):=op(a)^{k}op(a_{\theta}^{z-k})$,
where $k\in\mathbb{Z}$ is such that $Re(z)<k$.

Now let $a\in\Gamma_{cl}^{m}(\mathbb{R}^{n},\mathcal{B}(\mathbb{C}^{q}))$,
$m>0$, be a $\Lambda$ and $\tilde{\Lambda}$-elliptic pseudodifferential
operator that satisfies the assumption $(A)$ for both sectors and let
$\theta$ and $\theta'$ be angles whose rays are contained in the
sectors $\Lambda$ and $\tilde{\Lambda}$, respectively. We define
the sectorial projection of $a$ associated with the angular sector
$\Lambda_{\theta,\theta'}$ as the operator\[
\Pi_{\theta,\theta'}(op(a))=op(a)\left(\frac{1}{2\pi i}\int_{\Gamma_{\theta,\theta'}}\lambda^{-1}(op(a)-\lambda)^{-1}d\lambda\right).\]

The symbol of this operator will be denoted by $\Pi_{\theta,\theta'}(a)\in\Gamma_{cl}^{0}(\mathbb{R}^{n},\mathcal{B}(\mathbb{C}^{q}))$.
It is clear that \[
\Pi_{\theta,\theta'}(a)=a\sharp\left(\frac{1}{2\pi i}\int_{\Gamma_{\theta,\theta'}}\lambda^{-1}(a-\lambda){}^{-\sharp}d\lambda\right),\]
and 

\[
a_{\theta}^{z}=a^{\sharp k}\sharp\left(\frac{1}{2\pi i}\int_{\Gamma_{\theta}}\lambda_{\theta}^{z-k}(\lambda-a)^{-\sharp}d\lambda\right),\]
where $Re(z)<k$ and $a^{\sharp k}$ denotes $a\sharp...\sharp a$,
$k$ times.
\end{defn}

\begin{defn}
An idempotent classical Shubin pseudodifferential operator is a pseudodifferential
operator $op(a)$ with $a\in\Gamma_{cl}^{0}(\mathbb{R}^{n},\mathcal{B}(\mathbb{C}^{q}))$
or $a\in\Gamma^{-\infty}(\mathbb{R}^{n},\mathcal{B}(\mathbb{C}^{q}))$,
such that $op(a)op(a)=op(a)$.
\end{defn}
Using the same arguments of Seeley \cite{Seeley}, Shubin \cite{Shubin}
and Ponge \cite{Ponge}, we can prove the following Theorem.
\begin{thm}
\label{thm:propriedades complex and projection} Let $a\in\Gamma_{cl}^{m}(\mathbb{R}^{n},\mathcal{B}(\mathbb{C}^{q}))$,
$m>0$, be a $\Lambda$-elliptic pseudodifferential operator that
satisfies the assumption $(A)$ and $\theta$ be an angle contained
in $\Lambda$. Then $a^{z}$ is well defined, that is, for $Re(z)\ge0$
it does not depend on the choice of $k\in\mathbb{Z}$ such that $Re(z)<k$. It is
additive, that is $op(a_{\theta}^{z})op(a_{\theta}^{s})=op(a_{\theta}^{z+s})$,
for $z$ and $s\in\mathbb{C}$.

For $k\in\mathbb{Z}\backslash\{0\}$ we have that $op(a_{\theta}^{k})=op(a)^{k}$.
For $k=0$, we have $op(a_{\theta}^{0})=op(a^{-1})op(a)=I-op(\Pi_{0}(a))$, where $\Pi_{0}(a)\in \Gamma^{-\infty}(\mathbb{R}^{n},\mathcal{B}(\mathbb{C}^{q}))$ and $op(\Pi_{0}(a))$ is
an idempotent operator that has its image in the kernel of $op(a)$.

The sectorial projections of $a$ are idempotent operators that belong
to $\Gamma_{cl}^{0}(\mathbb{R}^{n},\mathcal{B}(\mathbb{C}^{q}))$.
\end{thm}
Using our estimates for the symbol of $a_{\theta}^{z}$ we can prove
that $a_{\theta}^{z}$ is a holomorphic family of pseudodifferential
operators. The following definition can be found in a similar version
in Kumano-go \cite{Kumanogo} and in Maniccia, Schrohe and Seiler
\cite{Schrohecomplex}.
\begin{defn}
Let $\Omega\subset\mathbb{C}$ be an open set. We say that $\{b^{z}\in\Gamma_{cl}^{\mathbb{C}}(\mathbb{R}^{n},\mathcal{B}(\mathbb{C}^{q}));\, z\in\Omega\}$
is a holomorphic family of classical pseudodifferential operators
when the following conditions are satisfied:\end{defn}
\begin{itemize}
\item There is an analytic function $\tau:\Omega\to\mathbb{C}$ such that
$b^{z}\in\Gamma_{cl}^{\tau(z)}(\mathbb{R}^{n},\mathcal{B}(\mathbb{C}^{q}))$.
\item If $b^{z}$ has the asymptotic expansion $b^{z}\sim\sum_{j}b_{(\tau(z)-j)}^{z}$,
then \[
(x,\xi,z)\mapsto b_{(\tau(z)-j)}^{z}(x,\xi)\in C^{\infty}\left(\left(\mathbb{R}^{n}\times\mathbb{R}^{n}\right)\backslash\{0\}\times\mathbb{C},\mathcal{B}(\mathbb{C}^{q})\right),\]
and it is analytic in the $z$ variable for each $(x,\xi)\in \mathbb{R}^n\times \mathbb{R}^n\backslash \{0\}$.
\item For any $z_{0}\in\Omega$, $N\in\mathbb{N}_{0}$ and $\epsilon>0$,
there is a neighborhood of $z_{0}$, denoted by $V_{0}\subset\Omega$,
such that $Re(\tau(z))<Re(\tau(z_{0}))+\epsilon$ for all $z\in V_{0}$.
Moreover for all zero excision function $\chi\in C^{\infty}(\mathbb{R}^{n}\times\mathbb{R}^{n})$,
the map $z\in V_{0}\mapsto b^{z}-\sum_{j=0}^{N-1}\chi b_{(\tau(z)-j)}^{z}\in\Gamma^{\tau(z_{0})-N+\epsilon}(\mathbb{R}^{n},\mathcal{B}(\mathbb{C}^{q}))$
is holomorphic.
\end{itemize}
If $\{b^{z}\in\Gamma_{cl}^{\mathbb{C}}(\mathbb{R}^{n},\mathcal{B}(\mathbb{C}^{q}));\, z\in\Omega\}$
and $\{c^{z}\in\Gamma_{cl}^{\mathbb{C}}(\mathbb{R}^{n},\mathcal{B}(\mathbb{C}^{q}));\, z\in\Omega\}$
are two holomorphic family of classical pseudodifferential operators,
then the above definition implies that its composition $\{b^{z}\sharp c^{z}\in\Gamma_{cl}^{\mathbb{C}}(\mathbb{R}^{n},\mathcal{B}(\mathbb{C}^{q}));\, z\in\Omega\}$
is also such a family.

From the definition of holomorphic families and of $a_{\theta}^{z}$, one can be prove the following proposition.
\begin{prop}
The complex powers of a classical pseudodifferential operator form
a holomorphic family of pseudodifferential operators $\{a_{\theta}^{z}\in\Gamma_{cl}^{\mathbb{C}}(\mathbb{R}^{n},\mathcal{B}(\mathbb{C}^{q}));\, z\in\mathbb{C}\}$,
where $\tau(z)=mz$.\end{prop}

\section{The traces of Kontsevich-Vishik and Wodzicki. The $\zeta$ function.}

In this section we study and define two linear functionals in the
algebras of classical pseudodifferential operators: The Wodzicki trace
and Kontsevich-Vishik functional for Shubin symbols. We remember that
a trace in an algebra $\mathcal{A}$ is a linear functional $\tau:\mathcal{A}\to\mathbb{C}$
that is zero on each commutator, $\tau(xy-yx)=0$ for all $x,y\in\mathcal{A}$.

The Wodzicki trace was first defined on compact manifolds by Wodzicki
in 1984 \cite{Wodzicki2}. In the context of Weyl's algebra it was
also independently discovered by Guillemin in 1985 \cite{Guillemin}.
For Shubin's operators, it was
defined -actually for a bigger class of operators- by Boggiatto and Nicola \cite{Boggiatto}, see also Section 5.1 of \cite{NicolaRodino}.
\begin{thm}
\label{thm:wodzicki trace} There exists a unique trace on the algebra $\Gamma_{cl}^{\mathbb{Z}}(\mathbb{R}^{n},\mathcal{B}(\mathbb{C}^{q}))/\Gamma^{-\infty}(\mathbb{R}^{n},\mathcal{B}(\mathbb{C}^{q}))$,
where $\Gamma_{cl}^{\mathbb{Z}}(\mathbb{R}^{n},\mathcal{B}(\mathbb{C}^{q}))=\cup_{m\in\mathbb{Z}}\Gamma_{cl}^{m}(\mathbb{R}^{n},\mathcal{B}(\mathbb{C}^{q}))$,
called the Wodzicki trace for Shubin operators. It is given by\[
\mbox{Res}(a)=\frac{1}{(2\pi)^n}\int_{S^{2n-1}}Tr(a_{(-2n)}(x,\xi))dxd\xi,\]
where $a_{(-2n)}$ is the $-2n$ term of the asymptotic expansion
of $a$ and $Tr$ denotes the usual trace on $\mathcal{B}(\mathbb{C}^{q})$.
\end{thm}
The proof of the uniqueness does not require any assumptions about
the topology of the algebra. It is a purely algebraic result. Nevertheless
it is clear that $\mbox{Res}|_{\Gamma_{cl}^{m}(\mathbb{R}^{n},\mathcal{B}(\mathbb{C}^{q}))/\Gamma^{-\infty}(\mathbb{R}^{n},\mathcal{B}(\mathbb{C}^{q}))}\to\mathbb{C}$
is continuous for any $m\in\mathbb{Z}$, where $\Gamma_{cl}^{m}(\mathbb{R}^{n},\mathcal{B}(\mathbb{C}^{q}))$
has the topology previously defined.

The other interesting linear functional we use is the Kontsevich-Vishik
functional, also called Kontsevich-Vishik trace although it is not
exactly a trace in our sense. It was first defined in \cite{Vishik}
for classical operators on compact manifolds. In order to keep the
notation simpler, we define this functional only for scalar symbols.
For matricial symbols, all we have to do is to insert the usual trace
$Tr$ of $\mathcal{B}(\mathbb{C}^{q})$ before each symbol. 

Let us fix a zero excision function $\chi$. We know that for every $a\in\Gamma_{cl}^{z}(\mathbb{R}^{n})$, there
exists $a_{(z-j)}\in\Gamma^{(z-j)}(\mathbb{R}^{n})$ and $r_{z-p}\in\Gamma_{cl}^{z-p}(\mathbb{R}^{n})$, depending
on $a$, uniquely determined,
such that \[
a-\sum_{j=0}^{p-1}\chi a_{(z-j)}=r_{z-p}.\]

Using this convention we define the function $TR$, which will be
the Kontsevich-Vishik functional for the Shubin class, also called finite-part
integral in the work of Maniccia, Schrohe and Seiler \cite{DeterminantesSG}.
\begin{defn}
We define $TR:\Gamma_{cl}^{z}(\mathbb{R}^{n})\to\mathbb{C}$, for
$z\in\mathbb{C}\backslash\{-2n,-2n+1,-2n+2,...\}$, in the following
way. We choose $p\in\mathbb{N}_{0}$ such that $z-p<-2n$ and define
for $a\in\Gamma_{cl}^{z}(\mathbb{R}^{n})$\foreignlanguage{brazil}{\[
TR(a)=\frac{1}{(2\pi)^{n}}\left\{ \int_{B_{1}(0)}a(x,\xi)dxd\xi-\sum_{j=0}^{p-1}\frac{1}{2n+z-j}\int_{S^{2n-1}}a_{(z-j)}(x,\xi)d\theta+\int_{B_{1}(0)^{c}}r_{z-p}(x,\xi)dxd\xi\right\} ,\]
}
\end{defn}
where $d\theta$ is the volume measure on $S^{2n-1}$.

The idea behind this definition is to {}``integrate'' the function
$a$. The function $TR$ is well defined, in the sense that $TR(a)$
does not depends on the choice of the zero excision function and on $p$. Moreover for each $z\in\mathbb{C}\backslash\{-2n,-2n+1,-2n+2,...\}$,
$TR:\Gamma_{cl}^{z}(\mathbb{R}^{n})\to\mathbb{C}$ is linear and continuous.
If $a\in\Gamma_{cl}^{z}(\mathbb{R}^{n})$ with $Re(z)<-2n$, then
$TR(a)$ coincides with the trace of $op(a)$, in the sense of trace
of a trace class operator acting in the Hilbert space $L^{2}(\mathbb{R}^{n})$.
It is given by $TR(a)=\frac{1}{(2\pi)^{n}}\int a(x,\xi)dxd\xi$.

This functional has the very nice property of producing analytic functions
from holomorphic families. Let us make this more precise.
\begin{prop}
Let $\{b^{z}\in\Gamma_{cl}^{\mathbb{C}}(\mathbb{R}^{n});\, z\in\Omega\}$
be a holomorphic family of classical Shubin pseudodifferential operators.
Then $\Omega\backslash\{z\in\mathbb{C};\tau(z)=-2n+j,\,\mbox{where}\, j\in\mathbb{N}_{0}\}\ni z\mapsto TR(b^{z})$
is a holomorphic function.\end{prop}
\begin{proof}
The proof follows from the definition of holomorphic families of classical
pseudodifferential operators and the explicit expression for $TR(b^{z})$:
\[
TR(b^{z})=\frac{1}{(2\pi)^n}\left\{\int_{|(x,\xi)|\le1}b^{z}(x,\xi)dxd\xi-\sum_{j=0}^{p-1}\frac{1}{2n+\tau(z)-j}\int_{S^{2n-1}}b_{(\tau(z)-j)}^{z}(1,\theta)d\theta+\int_{1\le|(x,\xi)|}r_{\tau(z)-p}^{z}(x,\xi)dxd\xi\right\},\]
where $p$ is such that $\tau(z)-p<-2n$ and $b^z=\sum_{j=0}^{p-1}\chi b^z_{(\tau(z)-j)}+r^p_{\tau(z)-p}$.
\end{proof}
We can relate the two functionals by the following proposition. Similar considerations were also used in \cite{Ubertino}.
\begin{prop}
\label{pro:residuo de TR(q sharp a -z)} Let $\{b^{z}\in\Gamma_{cl}^{\mathbb{C}}(\mathbb{R}^{n});\, z\in\Omega\}$
be a holomorphic family of classical pseudodifferential operators
defined on a neighborhood $\Omega$ of 0. If $b^{0}=b\in\Gamma_{cl}^{N}(\mathbb{R}^{n})$,
where $N\in\mathbb{Z}$ and $\tau(z):=ord(b^{z})=z+N$, then on $z=0$
the function $TR(b^{z})$ has at worst a simple pole and $res_{z=0}TR(b^{z})=-\mbox{Res}(b)$.
In particular we have $res_{z=0}TR(q\sharp a_{\theta}^{-z})=\frac{1}{m}\mbox{Res}(q)$,
for any $q\in\Gamma_{cl}^{\mathbb{Z}}(\mathbb{R}^{n})$ and $a\in\Gamma_{cl}^{m}(\mathbb{R}^{n})$
that satisfies the conditions for the construction of complex powers.\end{prop}
\begin{proof}
Using the expression for $TR(b^{z})$, with $\tau(z)=z+ord(b)$ we
see that near 0 we have to consider just the term with $j=2n+ord(b)$.
We then have \[
TR(b^{z})=-\frac{1}{(2\pi)^n z}\int_{S^{2n-1}}b_{(z-2n)}^{z}(x,\xi)d\theta+\mbox{analytic function near 0.}\]

Therefore $res_{z=0}TR(b^{z})=-\mbox{Res}(b^{0})=-\mbox{Res}(b).$

Finally let us consider the family $\{q\sharp a_{\theta}^{\frac{z}{m}},z\in\mathbb{C}\}$.
By theorem \ref{thm:propriedades complex and projection} $q\sharp a_{\theta}^{0}=q-q\sharp\Pi_{0}(a)$
and $ord(q\sharp a_{\theta}^{\frac{z}{m}})=z+ord(q)=z+ord(q-q\sharp\Pi_{0}(a))$.
Therefore the residue of the function $f(z)=TR(q\sharp a_{\theta}^{\frac{z}{m}})$
is $-\mbox{Res}(q-q\sharp\Pi_{0}(a))=-\mbox{Res}(q)$ and $res_{z=0}TR(q\sharp a_{\theta}^{-z})=res_{z=0}f(-mz)=\frac{1}{m}\mbox{Res}(q)$.
\end{proof}
We end this section using the Kontsevich-Vishik functional to define
the $\zeta$ function.
\begin{defn}
Let $a\in\Gamma_{cl}^{m}(\mathbb{R}^{n})$, with $m>0$ be $\Lambda$-elliptic.
Let $\theta$ be an angle whose corresponding ray belongs to $\Lambda$.
Then we can define complex powers of $a$, $a_{\theta}^{z}$, using
$\Gamma_{\theta}$ as the path of integration. The zeta function of
$a$ is defined on $\{z\in\mathbb{C};\, z\ne\frac{2n-j}{m}\,\mbox{for}\, j\in\mathbb{N}_{0}\}$
by \[
\zeta_{\theta}(a,z)=TR(a_{\theta}^{-z}).\]

\end{defn}
Explicitly the $\zeta_{\theta}$ function can be written as\foreignlanguage{brazil}{\begin{equation}
\zeta_{\theta}(a,z)=\frac{1}{(2\pi)^{n}}\left\{ \int_{B_{1}(0)}a_{\theta}^{-z}(x,\xi)dxd\xi-\sum_{j=0}^{p-1}\frac{1}{2n-mz-j}\int_{S^{2n-1}}a_{\theta(-mz-j)}^{-z}(x,\xi)d\theta+\right.\label{eq:zeta}\end{equation}
\[
\left.\int_{B_{1}(0)^{c}}r_{\theta-mz-p}^{-z}(x,\xi)dxd\xi\right\} ,\]
where $-mz-p<-2n$ and \[
a_{\theta}^{-z}(x,\xi)=\sum_{j=0}^{p-1}\chi(x,\xi)a_{\theta(-mz-j)}^{-z}(x,\xi)+r_{\theta-mz-p}^{-z}(x,\xi),\]
where $\chi\in C^{\infty}(\mathbb{R}^{n}\times\mathbb{R}^{n})$ is
a zero excision function.}

Using formula \ref{eq:zeta}, we conclude that the $\zeta_{\theta}$
function of a $\Lambda$-elliptic symbol $a\in\Gamma_{cl}^{m}(\mathbb{R}^{n})$,
with $m>0$, is a meromorphic extension with at worst simple poles
at $\{z=\frac{2n-j}{m}\,\mbox{for}\, j\in\mathbb{N}_{0}\}$
of the analytic function on $Re(z)>\frac{2n}{m}$, given by \[
tr(op(a_{\theta}^{-z}))=\frac{1}{\left(2\pi\right)^{n}}\int a_{\theta}^{-z}(x,\xi)dxd\xi,\]
where $tr$ is the trace of trace class operators in $L^{2}(\mathbb{R}^{n})$.
Moreover the poles are given by $\frac{1}{m}\mbox{Res}(a_{\theta}^{-\frac{2n-j}{m}})$.
\begin{rem}
Suppose that $a\in\Gamma_{cl}^{m}(\mathbb{R}^{n})$ is an elliptic
self-adjoint symbol and $m>0$. Then $A=op(a):\mathcal{Q}^{m}(\mathbb{R}^{n})^{\oplus q}\subset L^{2}(\mathbb{R}^{n})^{\oplus q}\to L^{2}(\mathbb{R}^{n})^{\oplus q}$
is an unbounded self-adjoint operator. Suppose in addition that $A$
is positive and let $\{\lambda_{j};\, j\in\mathbb{N}\}$ be the set
of eigenvalues of $A$. Then for $Re(z)>\frac{2n}{m}$ \[
\zeta_{\theta}(a,z)=\sum_{j=1}^{\infty}\lambda_{j}^{-z}.\]
This follows from Shubin \cite{Shubin}. \end{rem}
\begin{example}
Let us consider the operator, whose symbol is $a(x,\xi)=\frac{x^{2}+\xi^{2}+1}{2}$.
In this case $a\in\Gamma_{cl}^{2}(\mathbb{R})$ is an elliptic symbol
and $op(a):Q^{2}(\mathbb{R})\subset L^{2}(\mathbb{R})\to L^{2}(\mathbb{R})$
is a positive operator, whose eigenvalues are $\lambda_{j}=j$ for
$j\in\mathbb{N}$ (this is a simple consequence of the harmonic oscillator
results, see for instance Theorem 2.2.3 of \cite{NicolaRodino}).

We can evaluate explicitly the zeta function of this operator. In
fact for $Re(z)>\frac{2n}{m}=1$, the zeta function is\[
\zeta_{\theta}(a,z)=\sum_{n=1}^{\infty}\lambda_{n}^{-z}=\sum_{n=1}^{\infty}\frac{1}{n^{z}}.\]

This is the well known Riemann zeta function.
\end{example}

\section{K-Theory.}

In this section we compute the $K$ groups of the classical 0 order
Shubin pseudodifferential operators in order to get some information about traces of
idempotent operators. This approach of K-theory of $C^*$-algebras was also used by Gaarde \cite{Gaarde} to deal with boundary problems.

As in the case of classical pseudodifferential operators on manifolds,
we have seen in Corollary \ref{thm:norma modulo compactos} that for
a symbol $a\in\Gamma_{cl}^{0}(\mathbb{R}^{n})$,
whose principal symbol is $a_{(0)}\in C^{\infty}((\mathbb{R}^{n}\times\mathbb{R}^{n})\backslash\{0\})$,
the following equality holds\[
\inf_{C\in\mathcal{K}(L^{2}(\mathbb{R}^{n}))}\left\Vert op(a)+C\right\Vert _{\mathcal{B}(L^{2}(\mathbb{R}^{n}))}=\sup_{|(x,\xi)|=1}|a_{(0)}(x,\xi)|.\]

It is easily seen that $\Gamma^{(0)}(\mathbb{R}^{n})$ is isomorphic
to $C^{\infty}(S^{2n-1})$. As $C^{\infty}(S^{2n-1})$ is dense in
$C(S^{2n-1})$, we conclude from the above relation that\[
\overline{\Gamma_{cl}^{0}(\mathbb{R}^{n})}/\mathcal{K}\cong C(S^{2n-1}).\]

$\overline{\Gamma_{cl}^{0}(\mathbb{R}^{n})}$ is the closure
of the set of operators $op(a)$ in $\mathcal{B}(L^{2}(\mathbb{R}^{n}))$,
with $a\in\Gamma_{cl}^{0}(\mathbb{R}^{n})$ and $\mathcal{K}=\mathcal{K}(L^{2}(\mathbb{R}^{n}))$.
Using these isomorphisms we can calculate the $K$ groups of these
$C^{*}$-algebras. We refer to \cite{Rordamkteoria} for definitions
and notations.
\begin{lem}
Let $X\ne\emptyset$ be a metric space. Consider the function $f:X\to\mathbb{C}$
given by $f(x)=1$ for all $x\in X$. If $K_{0}(C(X))=\mathbb{Z}$,
then $[f]_{0}$ is a generator of $K_{\text{0}}(C(X))$.\end{lem}
\begin{proof}
Let $p\in X$ an arbitrary point and $\hat{p}:C(X)\to\mathbb{C}$
be the $*$-homomorphism given by $\hat{p}(f)=f(p)$. As $K_{0}$
is a functor, we can construct the group homomorphism $K_{0}(\hat{p}):K_{0}(C(X))\to K_{0}(\mathbb{C})$.
As $K_{0}(\mathbb{C})=\mathbb{Z}$, and $K_{0}(\hat{p})(f)=[f(p)]_{0}=[1]_{0}$,
we conclude that $K_{0}(\hat{p})$ is a homomorphism from $\mathbb{Z}$
to $\mathbb{Z}$ that is surjective. Therefore it must be also bijective.
As $[1]_{0}$ is the generator of $K_{0}(\mathbb{C})$, $[f]_{0}$
should also be the generator $K_{0}(C(X))$.
\end{proof}
Using this lemma we obtain the following simple but important Theorem.
\begin{thm}
The K groups of the $C^{*}$-algebras $\overline{\Gamma_{cl}^{0}(\mathbb{R}^{n})}$
and $\overline{\Gamma_{cl}^{0}(\mathbb{R}^{n})}/\mathcal{K}$ are

\[
K_{i}\left(\overline{\Gamma_{cl}^{0}(\mathbb{R}^{n})}/\mathcal{K}\right)=\mathbb{Z},\, i=0,1.\]
\[
K_{i}\left(\overline{\Gamma_{cl}^{0}(\mathbb{R}^{n})}\right)=\left\{ \begin{array}{c}
\mathbb{Z},\, i=0\\
0,\, i=1\end{array}\right..\]

Furthermore the $K_{0}$-class of the identity operator $I=op(1)$
is a generator of \textup{$K_{0}\left(\overline{\Gamma_{cl}^{0}(\mathbb{R}^{n})}\right)$}.\end{thm}
\begin{proof}
We have $K_{i}\left(\overline{\Gamma_{cl}^{0}(\mathbb{R}^{n})}/\mathcal{K}\right)\cong K_{i}\left(C(S^{2n-1})\right)\cong\mathbb{Z}$
for $i=0$ and $1$ \cite{Rordamkteoria}. For $\overline{\Gamma_{cl}^{0}(\mathbb{R}^{n})}$
we use the canonical cyclic exact sequence associated to the short
exact sequence \ref{eq:sequencia exata}:\[
\begin{array}{ccccc}
K_{0}(\mathcal{K}) & \overset{K_{0}(\overline{i})}{\longrightarrow} & K_{0}(\overline{\Gamma_{cl}^{0}(\mathbb{R}^{n})}) & \overset{K_{0}(\overline{s})}{\longrightarrow} & K_{0}(C(S^{2n-1}))\\
\begin{array}{cc}
\delta_{1} & \uparrow\end{array} &  &  &  & \begin{array}{cc}
\downarrow & \delta_{0}\end{array}\\
K_{1}(C(S^{2n-1})) & \overset{K_{1}(\overline{s})}{\longleftarrow} & K_{1}(\overline{\Gamma_{cl}^{0}(\mathbb{R}^{n})}) & \overset{K_{1}(\overline{i})}{\longleftarrow} & K_{1}(\mathcal{K})\end{array}.\]

We know that $K_{0}(\mathcal{K})=\mathbb{Z}$, $K_{0}(C(S^{2n-1}))=\mathbb{Z}$,
$K_{1}(\mathcal{K})=0$, $K_{1}(C(S^{2n-1}))=\mathbb{Z}$, see \cite{Rordamkteoria}.
This exact sequence then becomes\[
\begin{array}{ccccc}
\mathbb{Z} & \overset{K_{0}(\overline{i})}{\longrightarrow} & K_{0}(\overline{\Gamma_{cl}^{0}(\mathbb{R}^{n})}) & \overset{K_{0}(\overline{s})}{\longrightarrow} & \mathbb{Z}\\
\begin{array}{cc}
\delta_{1} & \uparrow\end{array} &  &  &  & \begin{array}{cc}
\downarrow & \delta_{0}\end{array}\\
\mathbb{Z} & \overset{K_{1}(\overline{s})}{\longleftarrow} & K_{1}(\overline{\Gamma_{cl}^{0}(\mathbb{R}^{n})}) & \overset{K_{1}(\overline{i})}{\longleftarrow} & 0\end{array}.\]

The map $\delta_{1}:\mathbb{Z}\to\mathbb{Z}$ gives the Fredholm index
\cite[Proposition 9.4.2]{Rordamkteoria}. As we know from the Fedosov
formula \cite{BVFedosov}, see also \cite[Theorem 19.3.1]{Hormander3}, there is a matrix-valued Fredholm
operator of index 1. Hence $\delta_{1}$ is an isomorphism. We conclude
that $K_{1}(\overline{\Gamma_{cl}^{0}(\mathbb{R}^{n})})=0$ and $K_{0}(\overline{\Gamma_{cl}^{0}(\mathbb{R}^{n})})=\mathbb{Z}$.
The function $f(x)=1$ is a generator of $K_{0}(C(S^{2n-1}))$. Therefore
the $K_0$ class of the identity operator $I=op(1)$ is a generator of $K_{0}(\overline{\Gamma_{cl}^{0}(\mathbb{R}^{n})})$. 
\end{proof}
From the stability of the $K$-Theory, we immediately get:
\begin{cor}
Let $\overline{\Gamma_{cl}^{0}(\mathbb{R}^{n},\mathcal{B}(\mathbb{C}^{q}))}$
the closure of the algebra generated by the operators of the form
$op(a)$ in $\mathcal{B}(L^{2}(\mathbb{R}^{n})^{\oplus q})$, with
$a\in\Gamma_{cl}^{0}(\mathbb{R}^{n},\mathcal{B}(\mathbb{C}^{q}))$.
Then 

\[
K_{i}\left(\overline{\Gamma_{cl}^{0}(\mathbb{R}^{n},\mathcal{B}(\mathbb{C}^{q}))}/\mathcal{K}(L^{2}(\mathbb{R}^{n})^{\oplus q})\right)=\mathbb{Z},\, i=0,1,\]
\[
K_{i}\left(\overline{\Gamma_{cl}^{0}(\mathbb{R}^{n},\mathcal{B}(\mathbb{C}^{q}))}\right)=\left\{ \begin{array}{c}
\mathbb{Z},\, i=0\\
0,\, i=1\end{array}\right..\]

Moreover, the $K_{0}$-class of the identity operator $I=op(1)$, where $1$ indicates the unity matrix, on
$L^{2}(\mathbb{R}^{n})^{\oplus q}$ is a generator of $K_{0}\left(\overline{\Gamma_{cl}^{0}(\mathbb{R}^{n},\mathcal{B}(\mathbb{C}^{q}))}\right)$.
\end{cor}
Our main objective now is to prove that the Wodzicki residue is zero
for any idempotent classical Shubin pseudodifferential operator.
\begin{thm}
\label{thm:wodzicki trace =0000E9 zero}The Wodzicki trace of any
idempotent Shubin operator in $\Gamma_{cl}^{0}(\mathbb{R}^{n},\mathcal{B}(\mathbb{C}^{q}))$
is zero.\end{thm}
\begin{proof}
Let us first identify $\Gamma_{cl}^{0}(\mathbb{R}^{n},\mathcal{B}(\mathbb{C}^{q}))$
with the algebra $\mathcal{A}$ of operators $op(a)$ in $\mathcal{B}(L^{2}(\mathbb{R}^{n})^{\oplus q})$
with symbols in $\Gamma_{cl}^{0}(\mathbb{R}^{n},\mathcal{B}(\mathbb{C}^{q}))$.
We have seen in section 2.4 that this algebra is a Fréchet algebra
and it is closed under the holomorphic functional calculus. Therefore,
the inclusion $i:\Gamma_{cl}^{0}(\mathbb{R}^{n},\mathcal{B}(\mathbb{C}^{q}))\to\overline{\Gamma_{cl}^{0}(\mathbb{R}^{n},\mathcal{B}(\mathbb{C}^{q}))}$
provides an isomorphism \[
K_{0}(i):\, K_{0}\left(\Gamma_{cl}^{0}(\mathbb{R}^{n},\mathcal{B}(\mathbb{C}^{q}))\right)\to K_{0}\left(\overline{\Gamma_{cl}^{0}(\mathbb{R}^{n},\mathcal{B}(\mathbb{C}^{q}))}\right)\]
(in order to see that the inclusion is an isomorphism of the K groups,
see, for instance, \cite[Proposition 3, page 292]{Connes} and the references mentioned there). Hence $[I]_{0}=[op(1)]_{0}$
is a generator of the group $K_{0}\left(\Gamma_{cl}^{0}(\mathbb{R}^{n},\mathcal{B}(\mathbb{C}^{q}))\right)$
and it is isomorphic to $\mathbb{Z}$. We know that for each trace
$\tau$ in $\Gamma_{cl}^{0}(\mathbb{R}^{n},\mathcal{B}(\mathbb{C}^{q})))$, there exists a unique
homomorphism $K_{0}(\tau):\, K_{0}\left(\Gamma_{cl}^{0}(\mathbb{R}^{n},\mathcal{B}(\mathbb{C}^{q}))\right)\to\mathbb{C}$
such that \[
K_{0}(\tau)([p]_{0})=\tau(p).\]
We know that $K_{0}(\mbox{Res})([I]_{0})=[\mbox{Res}(1)]_{0}=0$,
because obviously all the terms of the asymptotic expansion of 1 are
0, except for the homogeneous term of order 0. The trace of Wodzicki
depends only on the term $-2n$ of the asymptotic expansion, as it
can be seen in the Theorem \ref{thm:wodzicki trace}. As $[I]_{0}$
is the generator of $K_{0}\left(\Gamma_{cl}^{0}(\mathbb{R}^{n},\mathcal{B}(\mathbb{C}^{q}))\right)$,
we conclude that $K_{0}(\mbox{Res})$ is 0. Hence Res$(p)=0$ for
any idempotent $p\in\Gamma_{cl}^{0}(\mathbb{R}^{n},\mathcal{B}(\mathbb{C}^{q}))$.
\end{proof}

\section{The regularity of the $\eta$ function.}

We are now in position of proving the regularity of the $\eta$ function.
\begin{defn}
Let $a\in\Gamma^{m}(\mathbb{R}^{n},\mathcal{B}(\mathbb{C}^{q}))$,
$m>0$, be a self-adjoint, elliptic symbol. The $\eta$ function associated
to $a$ is defined as

\[
\eta(op(a),z)=\sum_{\lambda\in\sigma(op(a))\backslash\{0\}}sgn(\lambda)|\lambda|^{-z}=tr(op(a)|op(a)|^{-(z+1)}),\,\,\,\, Re(z)>\frac{2n}{m},\]
where $\sigma(op(a))$ denotes the spectrum of $op(a):Q^{m}(\mathbb{R}^{n})^{\oplus q}\subset L^{2}(\mathbb{R}^{n})^{\oplus q}\to L^{2}(\mathbb{R}^{n})^{\oplus q}$ and $tr$ the trace of trace class operators on $\mathcal{B}(L^2(\mathbb{R}^n))$.

This function can be extended to a meromorphic function in $\mathbb{C}$
with at most simple poles at $s=\frac{2n-j}{m}$, $j\in\mathbb{N}_{0}$, as can be proved using the same arguments we used for the $\zeta$ function\end{defn}
\begin{prop}
Let $a\in\Gamma_{cl}^{m}(\mathbb{R}^{n},\mathcal{B}(\mathbb{C}^{q}))$,
$m>0$, be a self-adjoint, elliptic symbol. Then
for any sector $\Lambda$ contained in the upper half plane, $\{z\in\mathbb{C},\mbox{Im}(z)>0\}$,
or in the lower half plane, $\{z\in\mathbb{C},\mbox{Im}(z)<0\}$,
the symbol $a$ is $\Lambda$-elliptic and satisfies the condition
$(A)$.\end{prop}
\begin{proof}
As $a$ is self-adjoint, $a=a^{*}$, we conclude that $a_{(m)}(x,\xi)=a_{(m)}(x,\xi)^{\dagger}$,
where $\dagger$ indicates the adjoint of a matrix in $\mathcal{B}(\mathbb{C}^{q})$.
Therefore $a_{(m)}(x,\xi)$ is a self-adjoint matrix for each $(x,\xi)$
and therefore the spectrum of the matrix $a_{(m)}(x,\xi)$ is contained
in $\mathbb{R}$. Hence \[
spec(a_{(m)}(x,\xi))\cap\Lambda=\emptyset,\,\,\,\forall|(x,\xi)|\ne0.\]

This proves the $\Lambda$-ellipticity. The $(A)$ condition is a
direct consequence of the fact that the spectrum of $op(a)$ is contained in the real axis.
\end{proof}
Now we state two results that are in essence independent of the class
of symbols that we use. They establish the connection between the Wodzicki trace and the $\eta$ and $\zeta$ function.

\begin{prop}
Let a symbol $a\in\Gamma_{cl}^{m}(\mathbb{R}^{n},\mathcal{B}(\mathbb{C}^{q}))$,
$m>0$, be a self-adjoint, elliptic symbol. We have
the equality of meromorphic functions,

\[
\lim_{z\to0}\left(\zeta_{\theta}(a,z)-\zeta_{\theta'}(a,z)\right)=2i\pi Res(\Pi_{\theta,\theta'}(a)),\]
\[
\mbox{res}_{s=0}\eta(op(a),s)=\frac{i}{\pi}\lim_{z\to k}(\zeta_{\uparrow}(a,z)-\zeta_{\downarrow}(a,z)).\]

Hence $\mbox{res}_{s=0}\eta(op(a),s)=2i\pi Res(\Pi_{\theta,\theta'}(a))$.\end{prop}
\begin{proof}
The first equality was proved by Wodzicki \cite{Wodzicki2} and the second was first stated by Shubin \cite{Shubin}. Ponge also proves them with details in \cite{Ponge}. These results were obtained for pseudodifferential operators on compact manifolds. However the proofs work equally well also for Shubin operators as can be seen using our results and the proofs in Section 4 of \cite{Ponge}. 
\end{proof}
The above proposition, together with Theorem \ref{thm:wodzicki trace =0000E9 zero}
implies the regularity of the $\eta$ function:
\begin{thm}
Let $a\in\Gamma_{cl}^{m}(\mathbb{R}^{n},\mathcal{B}(\mathbb{C}^{q}))$,
$m>0$, be an elliptic symbol and self-adjoint symbol. Then the $\eta$ function associated to $a$
is regular at $0$.
\end{thm}

{\section*{Acknowledgements}  The author would like to thank Severino Melo and  Elmar Schrohe for fruitful discussions and the Brazilian agency CNPq (Processo número 142185/2007-8) for financial support.  

\bibliographystyle{amsplain}

\end{document}